\documentclass[12pt]{amsart}
\usepackage[utf8]{inputenc}
\usepackage[T1]{fontenc}
\usepackage[english]{babel}
\usepackage{amssymb}
\usepackage{mathrsfs}
\usepackage{graphicx}
\usepackage[all]{xy}
\usepackage{enumerate}
\RequirePackage{hyperref} % Required for customising links and the PDF
\hypersetup{pdfpagemode={UseOutlines},
bookmarksopen=true,
bookmarksopenlevel=0,
hypertexnames=false,
colorlinks=true, % Set to false to disable coloring links
citecolor=cyan, % The color of citations
linkcolor=blue, % The color of references to document elements (sections, figures, etc)
urlcolor=magenta, % The color of hyperlinks (URLs)
pdfstartview={FitV},
unicode,
breaklinks=true,
}

\thanks{The author thanks the Isaac Newton Institute for Mathematical Sciences for the support and hospitality during the programme {\em Operators, Graphs, Groups} when work on this note was undertaken. This work was supported by EPSRC grant no EP/Z000580/1}

\usepackage{geometry} 
\newtheorem{prop}{Proposition}[section]
\newtheorem{thm}[prop]{Theorem}
\newtheorem{lem}[prop]{Lemma}

\theoremstyle{definition}

\title[Nilpotent groups have polynomial filling]{Nilpotent groups have polynomially bounded homological filling invariants}
\author{Gabriel Pallier}
\date{March 2026}
\begin{document}
\bibliographystyle{alpha}
\maketitle

\begin{abstract}
    Gromov claimed, with a sketch of proof, that simply connected nilpotent Lie groups have polynomially bounded filling invariants. The literature establishes this, often with a stronger conclusion where the exponent of polynomiality is computed or estimated, for some classes of nilpotent groups, or ranges of filling degrees. We provide a proof, in part based on Gromov's hints, yielding at once (non-optimal) polynomial upper bounds on the homological filling invariants in every degree for all finitely generated nilpotent groups, or equivalently, for all simply connected nilpotent Lie groups having lattices. 
\end{abstract}

\section{Introduction}

The filling invariants quantify the finiteness properties of a group by measuring the hardness to fill cycles or spheres of a given volume by chains or balls in an appropriate model space (e.g. a contractible CW complex or Riemannian manifold) for this group.
In this note we are interested in integral homological filling invariants of finitely generated nilpotent groups. These groups have finite torsion subgroups, and torsion-free nilpotent groups have classifying spaces given as CW complexes with finitely many cells and the homotopy type of a compact nilmanifold \cite{MR28842}. If $\Gamma$ is a torsion-free nilpotent group and $X$ is the universal cover of such a $K(\Gamma,1)$ having cells in dimensions up to $n$, one defines for all $v \in \mathbf Z_{\geqslant 0}$ and $d \in \{2,\ldots, n \}$, the combinatorial filling volume
\[ \mathrm{cFV}_{\Gamma}^d(v) = \sup_{ c\in \mathrm{C}_{d-1}(X, \mathbf Z): \partial c = 0, \vert c \vert \leqslant v} \inf \{ \vert r \vert : r \in \mathrm{C}_d(X, \mathbf Z) : \partial r = c  \}
\]
where $\vert \cdot \vert$ is the $\ell^1$-norm. A variant denoted $\operatorname{FV}^d_\Gamma$ can be defined using lipschitz chains on $X$ once the cells have been given $\Gamma$-invariant flat metrics with the isometry type of polyhedra, or using lipschitz chains on the universal cover of the Riemannian nilmanifold, and these variants have the same asymptotic behavior when $v$ goes to infinity as a result of the Federer-Fleming theorem from \cite{MR123260} (see e.g. \cite[Theorem 6.6]{bader2026higherkazhdanpropertyunitary} and references there). When $d=2$, $\operatorname{cFV}_\Gamma^d$ or $\operatorname{FV}_\Gamma^d$ is closely related to the Dehn function defined in terms of group presentations.
The aim of this note is to prove the following.

\begin{thm}
\label{prop:high-nil-fill}
    Let $\Gamma$ be a finitely generated nilpotent group of cohomological dimension $n$ over $\mathbf Q$.
    Then the Dehn function and the homological filling functions $\operatorname{cFV}^d_\Gamma$ with integer coefficients are polynomially bounded for all $d \in \{2, \ldots, n\}$.
\end{thm}

Theorem~\ref{prop:high-nil-fill} appears to be known to some authors: it is claimed to be true by Gromov (with a few hints towards a proof) and Varopoulos in \cite[p.72]{GromovAI} and \cite[p.59]{Varopoulos} respectively. We found no reference containing a fully detailed proof of the general case, although it is almost covered by the reunion of \cite[Corollary 7.2]{RileyHC} (from which one can deduce\footnote{A few details on that: Combining \cite[Corollary 7.2]{RileyHC} with \cite[Theorem 4]{YoungFillNil} and the fact that nilpotent groups have contractible asymptotic cones \cite{PanCBN} we find that the higher Dehn functions $\delta^{d-1}_{\Gamma}$ are polynomially bounded. For $d>3$, $\delta_\Gamma^{(d-1)}$ has the same growth type as $\mathrm{{FV}}_\Gamma^d$, see \cite[p.19206]{ABDYDehn} and the references given there. Finally $\operatorname{cFV}_\Gamma^d$ and $ \operatorname{FV}_\Gamma^d$ have the same growth rate by the Federer-Fleming theorem.} the result for all $d\geqslant 4$) and \cite{WengerAsRank} (which states it with an explicit upper bound on the polynomial order under the additional assumption that the Malcev completion of $\Gamma$ is Carnot-graded). For $d \in \{2, n\}$ and the Dehn function, optimal upper bounds are established, and these are even exact estimates when $d=n$ \cite{GHR, WengerAsRank, CSCIsop}. (When $\Gamma$ is the Heisenberg group, one has $n=3$, and the explicit estimates for $d=2,3$ were first obtained in \cite{MR1161694} and \cite{MR676380} respectively.)
Young, and later Gruber, proved that some nilpotent groups have Euclidean filling (that is, the same as $\mathbf Z^n$, and in particular, polynomial) in degrees $d$ below some threshold \cite{YoungFillNil,MR3472846,MR4002271,MR3982578}; for these groups their works also provide precise polynomial estimates of the filling invariants in the remaining degrees. Although always polynomially bounded, the Dehn function of nilpotent groups is not always exactly polynomial \cite{MR2783380}; this ``pathology'' has not been found in higher degrees so far.

%\footnote{For real, complex and quaternionic Heisenberg groups, an interesting phenomenon appears where the exponents of filling in degrees over the aforementionned threshold appear to be again those of euclidean spaces, but for one dimension higher.}.

Our motivation for proving Theorem~\ref{prop:high-nil-fill} is twofold. First, we observe that none of the references cited above provides a complete proof when $\Gamma$ is a lattice in a non-Carnot group and $d=3$. Second, polynomial upper bounds on the filling invariants recently appeared as important assumptions in works of L\'opez Neumann and Paucar \cite{neumann2026quantitativepolynomialcohomologyapplications} on the polynomial cohomology of nilpotent groups (see also \cite{bader2026higherkazhdanpropertyunitary}).
In the proof below, in part inspired by Gromov's sketches, we treat the general case but we do not attempt to optimize the exponent of polynomiality of the upper bound on the filling function.

\subsection{Acknowledgement}
I thank R. Young for pointing out references to me, J. Paucar and A. L\'opez Neumann for useful discussions.

\section{Proof of Theorem~\ref{prop:high-nil-fill}}

Our proof has two parts; in the first one we make precise the ``polynomial similarity'' between nilpotent Lie groups and their Lie algebras evoked by Gromov \cite[5.A\textsubscript{2}]{GromovAI} and use it to treat the case of the Dehn function along the lines of the sketch given in \cite[pp.56-57]{GromovAI}. The bound obtained in this way is already not optimal, we include it as a training for the general case, that we address in the second part\footnote{In \cite{GromovAI} on page 72, Gromov provides a hint on how one could proceed in the general case building on \cite[5.A\textsubscript{5}]{GromovAI} but we stress that the proof we give here does not quite follow this path at this stage; one  essential difference is that we use the Euclidean dilation in the Lie algebra instead of what Gromov calls the ``fake dilation'' of the Lie group.}. For the general case we additionally use the Federer-Fleming deformation theorem, which is the only non-elementary ingredient we require overall.

Let $G$ be the real Malcev completion of a finite-index torsion-free subgroup $\Gamma_0$ of $\Gamma$. $G$ has dimension $n$ as a real Lie group. Let $s$ be the nilpotency class of $\Gamma_0$, hence of $G$. Equip the Lie algebra $\mathfrak g = \operatorname{Lie}(G)$ with a system of linear coordinates $(x_1,\ldots, x_n)$ dual to a basis $(e_1,\ldots, e_n)$ such that for any $1\leqslant i < j \leqslant n$, $[e_i,e_j] \in \operatorname{span}\{ e_k \colon k > j \}$; it is always possible to find such a basis by picking one such that the sequence of subspaces $V_j = \{ e_k: k >j \}$ refines the lower central series filtration. Let $h$ be the left-invariant Riemannian metric on $G$ which coincides with the flat euclidean metric $dx_1^2 + \cdots + dx_n^2$ at $T_1G \simeq \mathfrak g$; call the latter $h_0$.
Our first goal is to compare the Riemannian metric $\exp^\star h$, the Euclidean metric $h_0$ and the inner products they induce in the $d$-fold exterior powers $\Lambda^d T \mathfrak g$ for $d \in \{1, \ldots,  n\}$ (See e.g. \cite[1.7.5]{FedGMT} for the construction of these inner products).
We will actually need bounds in both directions.

\begin{lem}[Polynomial similarity]
\label{lem:bounding-simple-vectors}
With notation as above, there exists $R_d\in \mathbf R[t]$, nondecreasing on $[0,+\infty)$ and such that for any $d$-vector field $V = V_1 \wedge \cdots \wedge V_d$ on $\mathbf R^d$ we have, for all $x = (x_1,\ldots, x_n) \in \mathbf R^n \simeq \mathfrak g$,
\begin{equation}
\label{eq:bounding-simple-vectors}
\frac{\Vert V_{x} \Vert_{\Lambda^dh_0}}{R_d (\Vert x \Vert)}
    \leqslant \Vert  V_{x} \Vert_{\Lambda^d \exp^\star h} \leqslant R_d (\Vert x \Vert) \Vert  V_{x} \Vert_{\Lambda^d h_0}.
\end{equation}    
\end{lem}

\begin{proof}

The Baker-Campbell-Hausdorff series for $G$ only has finitely many non-vanishing terms because $G$ is nilpotent, and so the left-invariant vector fields $X_1,\ldots, X_n$ which coincide with $e_1, \ldots,e_n$ at the origin are pulled-back to the Lie algebra via the exponential as
\begin{equation}
\label{eq:bounding-simple-vectors-1}
     (\exp_G^\star X_i)(x_1,\ldots, x_n) = e_i + \sum_{j>i} P_{i,j}(x_1,\ldots, x_n) e_j 
\end{equation}
where $P_{i,j} \in \mathbf R[x_1,\ldots, x_n]$ are polynomials of total degrees at most $s -1$ (see e.g. \cite{LeDonneTripaldiCCLow} for many concrete examples of such calculations). 
Fix now an integer $d \in \{1,\ldots, n\}$. We will generalize the previous considerations to the $d$-areas. Denote by $I = \{i_1, \ldots, i_d \}$ a multi-index with $1 \leqslant i_1 < \cdots < i_d \leqslant n$ and equip the set of multiindices with the lexicographic order.  An orthonormal frame is given by the simple left-invariant $d$-vectors $X_I = X_{i_1} \wedge \cdots \wedge X_{i_d}$ and
\begin{equation}
     \operatorname{exp}^\star X_I = e_I + \sum_{J > I} P^{[d]}_{I,J}(x_1,\ldots, x_n) e_J, \notag
\end{equation}
where $P_{I,J}^{[d]}$ are again polynomials.
 This equation expresses that the transition matrix between two orthogonal frames of $\Lambda^d h_0$ and 
$\Lambda^d \exp^\star h$ is unipotent with its non-diagonal entries being polynomials in $(x_1, \ldots, x_n)$. Inverting this matrix, we check that its inverse has the same properties.
The bounds \eqref{eq:bounding-simple-vectors} follow from this observation for some (maybe not nondecreasing) polynomial $\widetilde R_d$ instead of $R_d$; we may finally replace $\widetilde R_d$ by a larger polynomial $R_d$ which is nondecreasing on $[0,+\infty)$ in order to reach the conclusion of the lemma.
\end{proof}

In addition to these estimates, we will need a classical lemma on uniform polynomial distortion in nilpotent groups that can be derived from the work of Guivarc'h or from that of Osin \cite{GuivNil,OsinNil}.

\begin{lem}[Polynomial distortion]
\label{lem:distortion}
    With notation as above, there exists $C\geqslant0$ and $L \geqslant 1$ only depending on $G$ and $h$ such that for every $g \in G$, $d_h(1,g)/L - C \leqslant \Vert \log g \Vert \leqslant C+ L d_h(1,g)^s$.
\end{lem}

\begin{proof}
    Guivarc'h \cite[p.344]{GuivNil} proves that there exists a direct sum decomposition $\mathfrak g  = \oplus_{i=1}^s W_i$ and a suitable choice of  norms $\Vert \cdot \Vert_i$ on $W_i$, such that if one denotes $\varphi(x) = \sum_{i=1}^s \Vert x_i \Vert_i^{1/i}$ and then $B_r = \{ x \in \mathfrak g : \varphi(x) \leqslant r \}$ for all $r >0$, we have, for some $\rho, k >0$ and for all $n \in \mathbf N$,
    $B_G(1,\rho)^n \subseteq \exp_G B_{kn} \subseteq B_G(1_G,1)^n $. By the Svarc-Milnor lemma the word distances over the generating sets $B_G(1_G, 1)$ and $B_G(1,\rho)$ are both quasiisometric to the Riemannian distance $d_h$, and so we have $C_0$, $L_0$ such that $d_h(1,g)/L_0 - C_0 \leqslant \varphi(\log g) \leqslant L_0 d_h(1,g) + C_0$.
    On the other hand, we have $r^{1/i} \leqslant 1 + r$ for all $r$ and so
    $\varphi(\log g) \leqslant C_1 + L_1 \Vert \log g \Vert$ for some $C_1 > 0$ and $L_1 \geqslant 1$. This proves the left inequality of the lemma. For the inequality on the right, note that there exists $C_2, L_2$ such that $\varphi(\log g)^{s} \geqslant L_2 \Vert \log g \Vert- C_2$ and so $(L_0 d_h(1,g) + C_0)^s \geqslant L_2 \Vert \log g \Vert - C_2$; after adjusting $L$ and $C$ suitably, this proves the inequality on the right.
\end{proof}

We now finish the proof of the Dehn function part of the statement. We work with the following definition of the Dehn function: 
\[ \delta_G(v) = \sup_{\gamma \in \operatorname{Lip}(S^1,G), \,\operatorname{length}_h(\gamma) \leqslant r} \; \inf_{\Delta \in \operatorname{Lip}(B^2, G), \Delta_{\mid S^1} = \gamma} \operatorname{Area}(\Delta).  \]
We refer to \cite[2.C.1]{CTD} and \cite[Section 5]{MR1967746} for the equivalence of $\delta_G$ (when it grows at least linearly) with the definition of the Dehn function involving group presentations. 

Let $\gamma$ be a lipschitz loop with length $v$ in $G$. 
Left-translate $\gamma$ so that it goes through $1_G$; this does not change its length. 
Then consider $\widehat \gamma := \log_G \circ \gamma$, and set $\widehat v = \operatorname{length}(\widehat \gamma)$; note that we have $\widehat \gamma \subseteq B(0,R)$, where $R \leqslant C + v^s$ by Lemma~\ref{lem:distortion}, and so $\widehat v \leqslant R_1(C+v^s)$, applying\footnote{The case $d=1$ of Lemma~\ref{lem:bounding-simple-vectors} can be deduced directly from \eqref{eq:bounding-simple-vectors-1} and inverting a triangular matrix with polynomial off-diagonal entries so the full strength of the lemma is in fact not needed here.} Lemma~\ref{lem:bounding-simple-vectors} with $d=1$. Build a lipschitz disk $\widehat \Delta$ with the union of all the geodesic segments from $0$ to $\widehat \gamma(t)$ for $0 \leqslant t \leqslant \operatorname{length(\widehat \gamma)}$. By the Cauchy-Schwarz inequality
\[ \operatorname{Area}(\widehat \Delta) = \int_0^{\widehat v} \frac{1}{2}  \vert\langle \widehat \gamma(t), \widehat \gamma'(t) \rangle \vert dt \leqslant \widehat v^2. \]
Consider the lipschitz disk $\Delta = \exp(\widehat \Delta)$.
Then by \eqref{eq:bounding-simple-vectors} we have $\operatorname{Area}(\Delta) \leqslant R_2(\widehat v) \operatorname{Area}(\widehat \Delta) \leqslant R_2(\widehat v)\widehat v^2$, which is a polynomial in $\widehat v$. Thus the filling function of loops in
$G$ (hence the Dehn function of $\Gamma$) is bounded by a polynomial function of $\widehat v$. Altogether,
\begin{equation*}
    \operatorname{Area}(\Delta) \leqslant R_2(C+v^s)R_1(C+v^s)^2 = O(v^N)
\end{equation*}
where $N = s \left( 2 \deg  R_1 +\deg R_2 \right)$ (this is not the optimal $N = s+1$ of \cite{GHR}).

We now come to the homological filling functions.
Consider the manifold $\mathcal N = G/\Gamma_0$. This is a compact nilmanifold, in particular, it is the total space of an iterated tori bundle. Hence, as can be shown by induction, $\mathcal N$ is a PL manifold, and especially, it is triangulable. Let $X$ denote the finite simplicial complex of dimension $n$ underlying a finite triangulation of $\mathcal N$, and let $\widetilde X$ be its universal cover.
Endow the $1$-skeleton of $\widetilde X$ with the simplicial distance.
We have two singular triangulations, that we will denote $\nu$ and $\tau$, of $G$ and $\mathbf R^n$ respectively, by $\widetilde X$, $\nu$ being obtained by lifting the triangulation of $\mathcal N$ by $X$ to universal covers, and $\tau$ by composing with $\log_G$.
The triangulation $\nu$ is uniform, in the sense that
\begin{enumerate}[(i)]
    \item the associated map $\nu:\widetilde X^{(0)} \to G$ is a $\Gamma_0$-equivariant quasiisometry, and
    \item there exists a constant $L' \geqslant 1$ such that for every simplex $\sigma$ of $\widetilde X$, we have that $\operatorname{mass} \nu(\sigma) \in [1/L', L']$, where the mass is measured with respect to $h$.
\end{enumerate}

Up to $G$-translating $\nu$ slightly, we will also assume that $1_G$ supports a $0$-simplex of $\nu$ and equivalently, that $0_{\mathbf R^n}$ supports a $0$-simplex of $\tau$. With these adjustments, we have that
\begin{itemize}
    \item[(i')] the map $\nu:\widetilde X^{(0)} \to G$ is a $L'$-bilipschitz quasiisometry.
\end{itemize}

The triangulation $\tau$ does not have the properties (i) and (ii) above, nevertheless, given a $d$-simplex $\sigma$ of $\widetilde X$ and letting $r$ being the smallest number such that $\tau(\sigma) \subseteq  B(0,r)$, we have that $\operatorname{mass}_{h_0} \tau(\sigma) \leqslant L'R_d(r)$ by applying Lemma~\ref{lem:bounding-simple-vectors}.

We will need the following lemma.

\begin{lem}[Euclidean dilation]
\label{lem:coning}
    Let $k \in \{1,\ldots, n-1\}$ and let $\mathbf c = \sum_{i=1}^m a_i \sigma_i$ be a lipschitz $k$-cycle in $\mathbf R^n$. Assume that $0 \in \operatorname{supp} \mathbf c$.
    Then there exists an integral lipschitz $(k+1)$-chain $\mathbf C$ such that $\operatorname{mass}(\mathbf C) \leq \operatorname{mass}(\mathbf c)\operatorname{diam}(\operatorname{supp} \mathbf c)$, $\partial \mathbf C = \mathbf c$ and $\operatorname{supp}(\mathbf C) \subseteq B(0, \operatorname{diam}(\operatorname{supp}(\mathbf c)))$.
\end{lem}

\begin{proof}
We prove the lemma by a rudimentary version of the ``coning'' technique which is similar to the argument we used for the Dehn function; we mention that this argument has been employed many times in more elaborate ways in the literature, and notably in \cite{MR2153909}.
    Identify $\Delta^{k-1}$, resp. $\Delta^k$ with the simplices with vertices at $(e_1, \ldots, e_k)$, resp. at $(0,e_1, \ldots, e_k)$ in $\mathbf R^k$.
    For $i \in \{1,\ldots, m\}$ define $\Sigma_i(x) = t \sigma_i(\overline x)$ whenever $x \in \Delta^{k+1}$ decomposes as $t \overline x$ with $\overline x \in \Delta^k$ (this decomposition is unique unless $x=0$, in which case the definition of $\Sigma_i(x)$ is unambiguous).
    Set $\mathbf C = \sum_{i=1}^m a_i \Sigma_i$. Then $\mathbf C$ has its support in $B(0,\operatorname{diam}(\operatorname{supp} \mathbf c))$ by convexity of the latter ball.
    A computation using $\partial \mathbf c = 0$ gives that $\partial \mathbf C = \mathbf c$, and it remains to bound the mass of $\mathbf C$.
    For this, note that by the area formula, for all $i$ from $1$ to $m$ and denoting by $\lambda_j$ the Lebesgue measure on $\Delta^j$ we have
    \begin{align*}
        \operatorname{mass}(\Sigma_i) & = \int_{\Delta^{k+1}} \Vert \Lambda^{k+1} d \Sigma_i(x) \Vert d\lambda_{k+1}(x) \\
        & \leq \int_0^{\mathrm{diam}(\mathrm{supp}(\mathbf c))} \frac{t}{\operatorname{diam}(\operatorname{supp} \mathbf c)} \int_{\Delta^{k}} \Vert \Lambda^{k} d \sigma_i(x) \Vert d\lambda_k(x) dt \\
        & \leq \operatorname{mass}(\mathbf \sigma_i)\operatorname{diam}(\operatorname{supp} \mathbf c).
    \end{align*}
    We get the inequality required in the conclusion of the lemma by summing over all $i$ for $i$ from $1$ to $m$.
\end{proof}

Let $v \geqslant 0$.
Let $c$ be a lipschitz singular $d$-cycle in $\widetilde X$ of mass at most $v$; by mass in $\widetilde X$ we mean the total number of simplices (of all dimensions) in the support. 
We call $c_\nu$ and $c_{\tau}$ the lipschitz cycles in $G$ and $\mathbf R^n$ respectively which correspond to $\nu$ and $\tau$ in the triangulations.

Start assuming that $c$ (and thus also $c_\nu$ and $c_\tau$) is connected.
Left translating $c$ by some $\gamma_0 \in \Gamma_0$ (which does not affect the $\widetilde X$-mass), also assume that $1_G \in \operatorname{supp}(c_\nu)$ and $0_{\mathbf R^n} \in \operatorname{supp}(c_\tau)$.
Since $c$ is connected and $\nu^{(0)}: \widetilde X^{(0)} \to G$ is a bilipschitz quasiisometry by (i'), $\operatorname{diam}_h(\operatorname{supp}(c_\nu))$ is bounded by a constant $L'$ times $v$, and then by the distortion Lemma~\ref{lem:distortion}, $\operatorname{supp}(c_\tau)$ is contained in a ball $B(0,R)$ with $R \leqslant C_3+ L_3v^s$ for some constants $C_3 \geqslant 0$ and $L_3 \geqslant 1$.

We now apply Lemma~\ref{lem:coning} to $c_\tau$, 
which provides a $(d+1)$-chain $C$ in $\mathbf R^n$ (beware that it may not be supported on $\tau$), that we push back to $G$ by the exponential map; we claim that the resulting lipschitz singular integral chain $C_G = \exp_{\star} C$ has
\begin{equation}
\label{eq:bounding-mass-cg}
     \operatorname{mass}_h(C_G) \leq KR_{d+1}(C_3+L_3v^s)R_d(C_3+L_3v^s)^2.
\end{equation}
for some constant $K$.
Indeed, since $0 \in \operatorname{supp}(\partial C)$, for any simplex $\Sigma$ of $C$ we have that $\operatorname{supp} \Sigma \subseteq B(0, \operatorname{diam} \operatorname{supp} C) \subseteq B(0,K_0v)$ for some constant $K_0 \geqslant 1$ (where we apply Lemma~\ref{lem:coning} along with the fact that $c$ is connected), and then
$
\operatorname{mass}(\exp_\star \Sigma) \leqslant K_0^2 R_{d+1}(C_3+L_3v^s) \operatorname{mass}(\Sigma)$.
Summing this over all $\Sigma$ and letting $K = 1+K_0^2$ we get the required bound \eqref{eq:bounding-mass-cg} .

Since $\nu$ is a uniform triangulation of $G$, by the Federer-Fleming theorem (See \cite[Theorem 2]{YoungFillNil} for a suitable formulation) we can deform $C_G$ into a $(d+1)$-chain $C_\nu$ with $C_\nu = C_G + \partial R$, $R$ is a $(d+2)$-chain of mass $\lesssim R_d(v) v$, and 
\begin{equation*}
    \operatorname{mass}(C_\nu) \lesssim \operatorname{mass}(C_G).
\end{equation*}
Using again that $\nu$ is uniform, we have that $C_\nu$ comes from a chain on $\widetilde X$ with mass bounded by a polynomial in $v$; precisely 
$C_\nu = \sum_{\sigma \in \widetilde X} a_\sigma \nu \circ \sigma$, where 

$$\sum_\sigma \vert a_\sigma \vert \leqslant \sum_{\sigma : \dim \sigma = d+1} \vert a_ \sigma \vert \leqslant L \operatorname{mass}_h C_\nu \lesssim R_{d+1} (v^s) R_d(v^s)^2.$$

Finally, $c$ might not be connected. 
Assume that its connected components have masses $\lambda_1 \operatorname{mass}(c), \ldots, \lambda_p \operatorname{mass}(c)$ where $p$ is the number of connected components and $\lambda_1+\cdots + \lambda_p = 1$.
There exist constants $K_1 >0$ and $M> 1$ such that for $1\leqslant j \leqslant p$, the $j$-th connected component of $c$ can be filled with a chain of mass at most $K_1 + (\lambda_j \operatorname{mass}(c))^M$.
Summing the chains that fill every connected component, we find a chain $C$ filling $c$ which has mass
\begin{align*}
    \operatorname{mass}(C) & \leq K_1p + \sum_{j=1}^p (\lambda_j \operatorname{mass}(c))^M  = K_1p + \left( \sum_j \lambda_j^M \right) \operatorname{mass}(c)^M\leq K_1v + v^M
\end{align*}
where we used the bound $p \leqslant v$ because every connected component has at least one simplex in it. This finishes the proof.
%\end{proof}

\section{Final remarks}

The part of the proof above showing that the Dehn function of a simply connected nilpotent Lie group $G$ is polynomially bounded does not require $G$ to have any lattice. This result is not new: it follows, for example, from the fact that the asymptotic cones is simply connected, and a result of Gromov; the details of the proof were given by Dru\c{t}u \cite[Theorem  4.1]{MR1902363}.
In the part concerning the homological filling functions, the presence of a lattice is required, nevertheless, one can note that $\operatorname{FV}^2_G$ is bounded above by the super-additive closure of $\delta_G$, the argument being similar to the way we went on from filling connected to non-connected chains, see \cite[Remark 12.C.3]{CTD} or \cite[Proposition 2.28]{brady2021homologicaldehnfunctionsgroups}.  
    Hence $\operatorname{FV}^2_G$ is polynomially bounded regardless of $G$ having lattices or not. 
It would be desirable to generalize Theorem~\ref{prop:high-nil-fill} with $G$ simply connected nilpotent Lie group replacing $\Gamma$, the filling volume $\operatorname{FV}$ instead of $\operatorname{cFV}$, and $d>2$. 

As a result of the use of unrefined inequalities at several steps, and {contrarily to the impression that the upper bounds coming out of our proof might convey} if one makes them explicit, we do not expect the exponents of $\operatorname{FV}^d_\Gamma$  to increase with $d$ for $\Gamma$ nilpotent.
Some evidence for an {opposite} phenomenon (i.e., a convergence of these exponents towards $1$ when the nilpotency class is fixed) can actually be gathered from the case of abelian $\Gamma$ where $\operatorname{FV}_\Gamma^d(v) \asymp v^{1+1/d}$, from the exact determinations of the filling invariants in some cases in \cite{YoungFillNil,MR3472846,MR3982578} using a more refined approach than ours, and from the fact that the top-dimensional filling function has polynomial order $1+1/N$ where $N$ is the order of polynomial growth \cite{CSCIsop}.

\bibliography{biblionil.bib}

\newcommand{\etalchar}[1]{$^{#1}$}
\begin{thebibliography}{ABDY13}

\bibitem[ABDY13]{ABDYDehn}
Aaron Abrams, Noel Brady, Pallavi Dani, and Robert Young.
\newblock Homological and homotopical {D}ehn functions are different.
\newblock {\em Proc. Natl. Acad. Sci. USA}, 110(48):19206--19212, 2013.

\bibitem[BKS21]{brady2021homologicaldehnfunctionsgroups}
Noel Brady, Robert Kropholler, and Ignat Soroko.
\newblock Homological {D}ehn functions of groups of type $fp_2$, 2021.

\bibitem[Bri02]{MR1967746}
Martin~R. Bridson.
\newblock The geometry of the word problem.
\newblock In {\em Invitations to geometry and topology}, volume~7 of {\em Oxf.
  Grad. Texts Math.}, pages 29--91. Oxford Univ. Press, Oxford, 2002.

\bibitem[BS26]{bader2026higherkazhdanpropertyunitary}
Uri Bader and Roman Sauer.
\newblock Higher kazhdan property and unitary cohomology of arithmetic groups,
  2026.

\bibitem[CSC93]{CSCIsop}
Thierry Coulhon and Laurent Saloff-Coste.
\newblock Isop\'erim\'etrie pour les groupes et les vari\'et\'es.
\newblock {\em Rev. Mat. Iberoamericana}, 9(2):293--314, 1993.

\bibitem[CT17]{CTD}
Yves Cornulier and Romain Tessera.
\newblock Geometric presentations of {Lie} groups and their {Dehn} functions.
\newblock {\em Publ. Math., Inst. Hautes {\'E}tud. Sci.}, 125:79--219, 2017.

\bibitem[Dt02]{MR1902363}
Cornelia Dru\c~tu.
\newblock Quasi-isometry invariants and asymptotic cones.
\newblock volume~12, pages 99--135. 2002.
\newblock International Conference on Geometric and Combinatorial Methods in
  Group Theory and Semigroup Theory (Lincoln, NE, 2000).

\bibitem[ECH{\etalchar{+}}92]{MR1161694}
David B.~A. Epstein, James~W. Cannon, Derek~F. Holt, Silvio V.~F. Levy,
  Michael~S. Paterson, and William~P. Thurston.
\newblock {\em Word processing in groups}.
\newblock Jones and Bartlett Publishers, Boston, MA, 1992.

\bibitem[Fed96]{FedGMT}
Herbert Federer.
\newblock {\em Geometric measure theory.}
\newblock Class. Math. Berlin: Springer-Verlag, repr. of the 1969 ed. edition,
  1996.

\bibitem[FF60]{MR123260}
Herbert Federer and Wendell~H. Fleming.
\newblock Normal and integral currents.
\newblock {\em Ann. of Math. (2)}, 72:458--520, 1960.

\bibitem[GHR03]{GHR}
S.~M. Gersten, D.~F. Holt, and T.~R. Riley.
\newblock Isoperimetric inequalities for nilpotent groups.
\newblock {\em Geom. Funct. Anal.}, 13(4):795--814, 2003.

\bibitem[Gro93]{GromovAI}
M.~Gromov.
\newblock Asymptotic invariants of infinite groups.
\newblock In {\em Geometric group theory, {V}ol.\ 2 ({S}ussex, 1991)}, volume
  182 of {\em London Math. Soc. Lecture Note Ser.}, pages 1--295. Cambridge
  Univ. Press, Cambridge, 1993.

\bibitem[Gru19a]{MR4002271}
Moritz Gruber.
\newblock Filling invariants of stratified nilpotent {L}ie groups.
\newblock {\em Math. Z.}, 293(1-2):39--79, 2019.

\bibitem[Gru19b]{MR3982578}
Moritz Gruber.
\newblock The growth of the first non-{E}uclidean filling volume function of
  the quaternionic {H}eisenberg group.
\newblock {\em Adv. Geom.}, 19(3):415--420, 2019.

\bibitem[Gui73]{GuivNil}
Yves Guivarc'h.
\newblock Croissance polynomiale et p\'eriodes des fonctions harmoniques.
\newblock {\em Bull. Soc. Math. France}, 101:333--379, 1973.

\bibitem[LDT22]{LeDonneTripaldiCCLow}
Enrico Le~Donne and Francesca Tripaldi.
\newblock A cornucopia of {C}arnot groups in low dimensions.
\newblock {\em Anal. Geom. Metr. Spaces}, 10(1):155--289, 2022.

\bibitem[LNP26]{neumann2026quantitativepolynomialcohomologyapplications}
Antonio López~Neumann and Juan Paucar.
\newblock Quantitative polynomial cohomology and applications to $\textrm
  l^p$-measure equivalence, 2026.

\bibitem[Mal49]{MR28842}
A.~I. Malcev.
\newblock On a class of homogeneous spaces.
\newblock {\em Izv. Akad. Nauk SSSR Ser. Mat.}, 13:9--32, 1949.

\bibitem[Osi01]{OsinNil}
D.~V. Osin.
\newblock Subgroup distortions in nilpotent groups.
\newblock {\em Comm. Algebra}, 29(12):5439--5463, 2001.

\bibitem[Pan82]{MR676380}
Pierre Pansu.
\newblock Une in\'egalit\'e{} isop\'erim\'etrique sur le groupe de
  {H}eisenberg.
\newblock {\em C. R. Acad. Sci. Paris S\'er. I Math.}, 295(2):127--130, 1982.

\bibitem[Pan83]{PanCBN}
Pierre Pansu.
\newblock Croissance des boules et des g\'eod\'esiques ferm\'ees dans les
  nilvari\'et\'es.
\newblock {\em Ergodic Theory Dynam. Systems}, 3(3):415--445, 1983.

\bibitem[Ril03]{RileyHC}
T.~R. Riley.
\newblock Higher connectedness of asymptotic cones.
\newblock {\em Topology}, 42(6):1289--1352, 2003.

\bibitem[Var00]{Varopoulos}
Nick~Th. Varopoulos.
\newblock A geometric classification of {L}ie groups.
\newblock {\em Rev. Mat. Iberoamericana}, 16(1):49--136, 2000.

\bibitem[Wen05]{MR2153909}
S.~Wenger.
\newblock Isoperimetric inequalities of {E}uclidean type in metric spaces.
\newblock {\em Geom. Funct. Anal.}, 15(2):534--554, 2005.

\bibitem[Wen11a]{WengerAsRank}
Stefan Wenger.
\newblock The asymptotic rank of metric spaces.
\newblock {\em Comment. Math. Helv.}, 86(2):247--275, 2011.

\bibitem[Wen11b]{MR2783380}
Stefan Wenger.
\newblock Nilpotent groups without exactly polynomial {D}ehn function.
\newblock {\em J. Topol.}, 4(1):141--160, 2011.

\bibitem[You13]{YoungFillNil}
Robert Young.
\newblock Filling inequalities for nilpotent groups through approximations.
\newblock {\em Groups Geom. Dyn.}, 7(4):977--1011, 2013.

\bibitem[You16]{MR3472846}
Robert Young.
\newblock High-dimensional fillings in {H}eisenberg groups.
\newblock {\em J. Geom. Anal.}, 26(2):1596--1616, 2016.

\end{thebibliography}
\end{document}